\documentclass{elsarticle}
\providecommand{\XLeftMargin}{2cm}
\providecommand{\XTopMargin}{1.8cm}
\providecommand{\XRightMargin}{2cm}
\providecommand{\XBottomMargin}{1.8cm}

\usepackage{amssymb,latexsym,amsmath,amsthm,amscd}
\usepackage[all]{xy}
\xyoption{arc}
\providecommand{\XLeftMargin}{2cm}
\providecommand{\XTopMargin}{1.8cm}
\providecommand{\XRightMargin}{2cm}
\providecommand{\XBottomMargin}{1.8cm}

\usepackage[left=\XLeftMargin,top=\XTopMargin,right=\XRightMargin,bottom=\XBottomMargin]{geometry}
\usepackage{graphicx}
\newlength\XXXMyLength\makeatletter
\@ifclassloaded{amsart}{\def\Xadjustleft[#1]{}
}{
\def\Xadjustleft[#1]{\setlength\XXXMyLength{#1}\ifnum\numexpr\leftmargin>\numexpr\XXXMyLength\hspace{-\XXXMyLength}\else\hspace{-\leftmargin}\fi}
}
\makeatother

\UseComputerModernTips

\theoremstyle{plain}
\newtheorem{thm}{Theorem}[section]
\newtheorem*{thm*}{Theorem}

\newtheorem*{lem*}{Lemma}
\newtheorem{prop}[thm]{Proposition}
\newtheorem*{prop*}{Proposition}

\newtheorem{cor}[thm]{Corollary}
\newtheorem*{cor*}{Corollary}
\newtheorem*{conv}{Convention}

\newtheorem*{conj*}{Conjecture}

\theoremstyle{definition}

\newtheorem*{cons*}{Construction}
\newtheorem{df}[thm]{Definition}
\newtheorem*{df*}{Definition}
\newtheorem{nota}[thm]{Notation}
\newtheorem*{nota*}{Notation}

\newtheorem*{qu*}{Question}

\newtheorem{rmk}[thm]{Remark}
\newtheorem*{rmk*}{Remark}

\newtheorem*{ex*}{Example}

\newcommand{\bC}{\mathbb{C}}

\newcommand{\bG}{\mathbb{G}}

\newcommand{\bO}{\mathbb{O}}

\newcommand{\bR}{\mathbb{R}}

\newcommand{\bZ}{\mathbb{Z}}

\DeclareMathOperator{\Aut}{Aut}

\DeclareMathOperator{\Gal}{Gal}

\DeclareMathOperator{\Res}{Res}

\DeclareMathOperator{\Stab}{Stab}

\DeclareMathOperator{\Tr}{Tr}

\DeclareMathOperator{\Gl}{GL}

\DeclareMathOperator{\Sl}{SL}
\DeclareMathOperator{\SO}{SO}

\DeclareMathOperator{\UU}{U}
\DeclareMathOperator{\SU}{SU}

\newcommand{\surjects}{\twoheadrightarrow}
\newcommand{\injects}{\hookrightarrow}

\def\emphh{\textbf}

\def\sumprime{\mathop{\sum{\raise3pt\hbox{${}'$}}}}

\makeatletter
\def\revddots{\mathinner{\mkern1mu\raise\p@
\vbox{\kern7\p@\hbox{.}}\mkern2mu
\raise4\p@\hbox{.}\mkern2mu\raise7\p@\hbox{.}\mkern1mu}}
\makeatother

\providecommand{\abs}[1]{\left\vert #1 \right\vert}

\newcommand{\comment}[1]{}

\newcommand{\GL}{\Gl}
\newcommand{\SL}{\Sl}

\begin{document}

\title{Classification of Certain Subgroups of G2}
\author{Andrew Fiori}

\ead{afiori@mast.queensu.ca}
\address{Department of Mathematics and Statistics, Queens University, Jeffery Hall, University Ave, Office 517, Kingston, ON, K7L 3N6, Canada}

\begin{abstract}
We give a concrete characterization of the rational conjugacy classes of maximal tori in groups of type G2, focusing on the case of number fields and p-adic fields.
In the same context we characterize the rational conjugacy classes of A2 subgroups of G2.
Having obtained the concrete characterization, we then relate it to the more abstract characterization which can be given in terms of Galois cohomology.

We note that these results on A2 subgroups were simultaneously and independently developed in the work of \cite{Hooda} whereas the results on tori were simultaneously and independently developed in the work of \cite{LeeGilleBeliPreprint}.
\end{abstract}


\maketitle

\section{Introduction}

A very natural problem in the theory of algebraic groups is to attempt to characterize or describe equivalences classes of maps:
\[ f : H \rightarrow G. \]
For reductive groups $H$ and $G$ over an algebraically closed field, the problem is at least in some sense well understood.
Consequently, when we instead wish to work over a field $k$ which is not algebraically closed, it is natural to consider instead the following problem.

Fix reductive algebraic groups $\tilde{H}_{\overline{k}}$ and $\tilde{G}_{\overline{k}}$ and a map $\tilde{f}_{\overline{k}} : \tilde{H}_{\overline{k}} \rightarrow \tilde{G}_{\overline{k}}$ all defined over the algebraic closure $\overline{k}$ of $k$.
One would like a concrete description of all triples:
\[ (H_k, G_k, f : H_k \rightarrow G_k) \]
consisting of reductive groups $H_k$, $G_k$ and a map between them, all defined over $k$,
which become equivalent to $\tilde{f}_{\overline{k}} : \tilde{H}_{\overline{k}} \rightarrow \tilde{G}_{\overline{k}}$ after base change to the algebraic closure.
Here equivalence over $k$ (respectively $\overline{k}$) means $G_k(k)$ (respectively, $\tilde{G}_{\overline{k}}(\overline{k})$) conjugacy.

The numerous special cases of the above general problem have applications in a variety of fields, especially number theory, where as a consequence of various functorialities such maps give results concerning the buildings of the groups, automorphic forms on the groups or special cycles on associated Shimura varieties.
In particular, when $G_k$ is associated to a Shimura variety, the characterization of maximal tori gives information about its special points (see \cite[3.15]{Deligne_Shimura1}).
Over $p$-adic fields, the structure of maximal tori gives information about the Bruhat-tits building of the group $G_k$. Moreover, combining expected functorialities in the Langlands correspondence and the local Langlands correspondence for tori, the structure of maximal tori gives information about the Langlands correspondence for $G_k$ and can be useful in constructing supercuspidal representations (see \cite{LawrenceToriClassical,MorrisToriSymplectic,RoeThesis}). 

Motivated by these and other reasons, several characterizations of maximal tori in reductive groups exist.
A very general, but not necessarily concrete description via Galois cohomology can be found in \cite{WalkerThesisTori} and \cite{ReederElliptic}.
In specific cases some aspects of this classification have been made concrete, \cite{WalkerThesisTori}, for example, gives a concrete interpretation for the classification  of the rational conjugacy classes of tori in pure inner forms of orthogonal and unitary groups.
The concrete problem of classifying when a given torus actually embeds in an orthogonal group is the main subject of my previous paper \cite{Fiori1} and is also discussed in \cite{Brus_OrthTori,BayerTori}.
A detailed study of the case of arbitrary forms of (special) unitary groups is the subject of my joint work with David Roe \cite{FioriRoe}. Other results related to the case of unitary groups can be found in, for example, \cite{CKMFrobeniusAlgebras}.
The question of a local-global principle relating the existence of local embeddings to the existence of global embeddings is looked at, for example in, \cite{ PR_localglobal, BayerTori}.
Other work on this type of problem includes \cite{KariyamaTori}.
 
In this paper we shall be working in the context of $k$ a number field or one of its completions, and we shall be taking $\tilde{G}_{\overline{k}}$ to be the (unique) group of type G2.
We shall be interested in two different cases for the group $\tilde{H}_{\overline{k}}$: the case where it is a simply connect group of type A2, and the case of $\bG_m^2$. In this later case we shall typically refer to $H_k$ as $T_k$ to indicate it is an algebraic torus. The injective maps $f$ shall be discussed in more detail later once we give an adequate description of the group G2. For now we indicate only that the torus $\bG_m^2$ shall be a maximal torus of $\tilde{G}_{\overline{k}}$ and its map into $\tilde{G}_{\overline{k}}$ shall factor through a unique simply connected A2 subgroup $\tilde{H}_{\overline{k}}$.
The main result of this paper is thus to give a characterization of the rational conjugacy classes of maximal tori in groups of type G2 over number fields. We also obtain auxiliary results concerning rational conjugacy classes of these A2 subgroups.

We note that these results on A2 subgroups were simultaneously and independently developed in the work of \cite{Hooda} whereas the results on tori were simultaneously and independently developed in the work of \cite{LeeGilleBeliPreprint}. This paper likely does not contain significant results not already found in one or the other of these. Our treatment may be helpful to non-experts, as it might be more elementary than and includes many background facts which experts will consider well-known. 

This paper is organized as follows:
\begin{itemize}
\item In section \ref{sec:octonions} we give the basic properties of the octonions that are relevant in the sequel.
\item In section \ref{sec:g2} we describe the automorphisms of the octonions and introduce the groups $G_k$.
\item In section \ref{sec:forms} we give results concerning the forms of algebraic groups, focusing on the cases relevant in the sequel.
\item In section \ref{sec:classifications} we prove our main results concerning the rational conjugacy classes of maximal tori.
\item In secti/on \ref{sec:groupcohom} we discuss the connection between our results and the abstract classification given in \cite{ReederElliptic} using Galois cohomology.
\end{itemize}

The main results concerning A2 subgroups are as follows.
\begin{thm*}[\ref{thm:SUclassification}]
Fix an octonion algebra $\bO$ over $k$ and let $G_k$ be the algebraic group associated to its automorphism group and set $\tilde{G}_{\overline{k}} = G_{\overline{k}}$.
Fix $x\in \bO\otimes_k\overline{k}$ perpendicular to $1_{\bO}$.
Fix $\tilde{H}_{\overline{k}} = \Stab_x(G_{\overline{k}})$ to be the stabilizer in $G_{\overline{k}}$ of $x$ and fix its natural inclusion:
\[ \tilde{f} : \tilde{H}_{\overline{k}} \injects \tilde{G}_{\overline{k}}. \]
Then the set of triples  $(H_k, G_k, f : H_k \rightarrow G_k)$, considered up to $G_k(k)$-conjugacy, that become $ \tilde{G}_{\overline{k}}(\overline{k})$-conjugate to $\tilde{f} : \tilde{H}_{\overline{k}} \injects \tilde{G}_{\overline{k}}$ after base change are in natural bijection with the isomorphism classes of quadratic \'etale algebras $A$ over $k$ that have an embedding $A\injects \bO$.
\end{thm*}
The groups $H_k$ of the above theorem are precisely the A2 subgroups we are considering.

\begin{cor*}[\ref{cor:SUoverR}]
Fix an octonion algebra $\bO$ over $k$ and let $G_k$ be the algebraic group associated to its automorphism group.
The isomorphism classes of simply connected A2 subgroups embedding in $G_k$ are controlled by the structure of $\bO$ at the real places of $k$.
In particular, fixing a real place $k_\nu$ of $k$ we have the following cases:
\begin{itemize}
\item If $\bO$ is definite at $k_\nu$ then all $A$ appearing are imaginary extensions of $k_\nu$ (the place $\nu$ ramifies in $A$) and the special unitary group $H_{k_\nu}$ is definite.
\item If $\bO$ is indefinite at $k_\nu$ and $A$ is an imaginary extensions of $k_\nu$ then the special unitary group $H_{k_\nu}$ is indefinite.
\item If $\bO$ is indefinite at $k_\nu$ and $A$ is a real extension of $k_\nu$ (the place $\nu$ splits in $A$), then the group $H_{k_\nu}$ is split, that is, $H_{k_\nu} \simeq \SL_{3,k_\nu}$.
\end{itemize}
\end{cor*}

Concerning maximal tori we conclude the following.
\begin{thm*}[\ref{thm:ClassifyToriG2}]
Fix an octonion algebra $\bO$ over $k$ and the group $G_k=\Aut_k(\bO)$.
The $G_k(k)$-conjugacy classes of maximal tori $T_k$ in $G_k$ are in bijection with triples $(A,E,\lambda)$ consisting of a quadratic \'etale algebra $A$ (with involutions $\sigma$) which embeds in $\bO$, a cubic \'etale algebra $E$ and an element $\lambda\in (E^\times)/N_{E \otimes_k A/E}((E \otimes_k A)^\times)\Aut_k(E) $ such that the $A$-Hermitian space $E\otimes_k A$ of dimension $3$ with Hermitian form:
\[ \Tr_{E\otimes_k A/A}( \lambda x \sigma(y) ) \]
has discriminant $1$ and is positive definite (respectively indefinite) at all the real places of $k$ where the octonion algebra is definite (respectively split).
Note that the form is positive definite if and only if $\lambda$ is totally positive and $E$ is totally real. Moreover, the discriminant of the form is $N_{E\otimes_k A/A}(\lambda)\delta_{E/k}$.

The torus $T_k$ associated to this data is precisely $T_{E\otimes_k A,\sigma,N}$ whose points over $R$ are: 
\[ T_{E\otimes_k A,\sigma,N}(R) = \{ x\in (E\otimes_k A \otimes_k R)^\times \mid x\sigma(x) = 1 \text{ and } N_{E\otimes_k A/A}(x) = 1 \} \]
where $\sigma$ is induced from the non-trivial automorphism of $A$.
\end{thm*}

\begin{prop*}[\ref{prop:ClassifyToriG2R}]
We have the following restrictions on the algebras $A$ and $E$ and the signature of $\lambda$ based on the structure of $\bO$ at each real place $\nu$.
The conditions can be summarized as follows:
\begin{itemize}
\item If $\bO$ is definite at $\nu$ then $A$ is a CM-algebra and $E$ is totally real. Moreover, $\lambda$ is positive at all the real places.
\item If $\bO$ is indefinite at $\nu$ then $A$ is arbitrary and $E$ is arbitrary. Furthermore, 
\subitem If $A$ is CM and $E$ is totally real then $\lambda$ is positive at a unique place.
\subitem If $A$ is CM and $E$ is not totally real then $\lambda$ is negative at the unique real place.
\subitem If $A$ is totally real then the choice of $\lambda$ at $\nu$ is irrelevant (since the norm map from $E\otimes_k A$ to $E$ is surjective at $\nu$).
\end{itemize}
\end{prop*}

\begin{cor*}[\ref{cor:ClassifyToriG2R}]
The $G_k(k)$-conjugacy classes of maximal tori $T_k$ in $G_k$ are in bijection with triples $(A,E,\lambda)$
where $A$ and $E$ are respectively quadratic and cubic extensions of $k$, the element $\lambda$ is in the kernel of the map:
\[ (E^\times)/N_{E \otimes_k A/E}((E \otimes_k A)^\times)/\Aut_k(E) \overset{N_{E/k}}\longrightarrow (k^\times)/N_{A/k}(A^\times) \]
and such that the triple $(A,E,\lambda\delta_{E/k})$ satisfies the conditions of Proposition \ref{prop:ClassifyToriG2R}.
\end{cor*}

\section{The Octonions}
\label{sec:octonions}

The purpose of this section is to introduce octonion algebras and some basic facts we shall need.
Most, if not all, of these basic facts are likely known to experts. Proofs or sketches of proofs are included for the benefit of non-experts.

{\def\XMetaCompile{1}
\def\section{\subsection}

\section{The Cayley-Dickson Construction}

For the purpose of this section we will be considering algebras over a fixed field $k$ of characteristic not equal to $2$. Much of what follows can be done over a more general ring, however, modifications should be made to deal with the characteristic $2$ case (see \cite[Sec. 33.C]{book_of_involutions}). It should be noted before we begin, that the algebras being considered are not necessarily commutative or even associative, though $k$, being a field, is both.

\begin{df}
Let $A$ be a $k$-algebra with involution $\sigma_A$ such that $N_A(x) = x\sigma(x)$ is a non-degenerate $k$-valued quadratic form on $A$. Fix $\delta\in k^\times$.
Consider the $k$-module $B=A\oplus A$. This can be given a $k$-algebra structure by defining:
\[ [x_1,x_2]\cdot[y_1,y_2] = [x_1y_1+\delta \sigma_A(y_2)x_2,y_2x_1+x_2\sigma_A(y_1)]. \]
We shall refer to this construction as the \emphh{Cayley-Dickson construction} and the algebra $B$ as a \emphh{CD-algebra}.

Note that there is a natural inclusion $A\injects B$ via $x\mapsto[x,0]$.
Moreover, $B$ comes equipped with a standard involution $\sigma_B$ defined as:
\[ \sigma_B([x_1,x_2]) = [\sigma_A(x_1),-x_2]. \]
The involution $\sigma_B$ restricts to $\sigma_A$ under the natural inclusion.

Every CD-algebra $B$ is equipped with a $k$-valued quadratic form given by:
\[  N_B(x) = x\sigma(x). \]
This is also called the \emphh{reduced norm} of $B$.

The associated bilinear pairing is $(x,y) \mapsto \tfrac{1}{2}\left(x\sigma(y) + y\sigma(x)\right)$.
Two elements $x,y\in B$ are said to be \emphh{orthogonal} or \emphh{perpendicular} if $(x,y) = 0$.
\end{df}

\begin{prop}
Consider the CD-algebra $B$ constructed from $A$ and $\delta$.
Every non-zero element of $B$ has a multiplicative inverse if and only if every element of $A$ has a multiplicative inverse and there is no $x \in A$ such that $\delta = x\sigma_A(x)$.
\end{prop}
\begin{proof}
The conditions that $A$ has multiplicative inverses and that $x\sigma_A(x)$ does not represent delta is equivalent to the condition that $N_B(x)$ does not represent $0$. In this case $\frac{\sigma(x)}{N_B(x)}$ will be a multiplicative inverse of a non-zero $x$.
If the norm form did represent zero then the algebra would have non-trivial zero divisors.
\end{proof}

\begin{prop}
The following properties of $A$ imply properties of $B$:
\begin{enumerate}
\item The algebra $B$ is commutative if and only if $A=k$.
\item The algebra $B$ is associative if and only if $A$ is commutative.
\item The algebra $B$ is alternative (i.e. $x(xy) = (x^2)y$ and $(yx)x = y(x^2)$) if and only if $A$ is associative.
\end{enumerate}
\end{prop}
See \cite[Lem. 33.16]{book_of_involutions}.
\begin{proof}[Sketch of proof]
To prove each claim one can reduced it to a simpler check involving only a basis of $B = A \oplus A$. We shall fix an orthogonal basis for $A$ which includes $1$, and use the same basis on both factors. The relations one needs to check are always obvious if one of the elements used is $1$, or all of them are in $A$.

One can readily check the following:
\begin{enumerate}
\item If $x$ is perpendicular to $1$ then $\sigma(x) = -x$.
\item If $1,x,y$ are perpendicular to each other, then $xy=-yx$.
\item If $x,y,z$ are mutually perpendicular basis vectors from the second $A$ factor then $(xy)z = -x(yz)$.
\item By the distributivity of products over sums, checking commutativity and associativity can be done on a basis.
\item To check alternativity one must additionally handle the case where $x$ is the sum of two basis vectors one of which is in $A$.
\item If $x,y$ are perpendicular basis vectors chosen as above, then checking $x(xy) = (xx)y$ involves only alternativity in $A$.
\item For $x$ the sum of two perpendicular basis vectors, one of which is in $A$, checking $x(xy) = (xx)y$ requires associativity of $A$.
\end{enumerate}
Using the above we may deduce the results of the proposition.
\end{proof}

\begin{nota}
Given a CD-algebra $B$, and an element $x\in B$  with $N_B(x)\in k^\times$ and which is orthogonal to $1$ we shall denote by $A_x = k\langle x\rangle$ the $k$-subalgebra generated by $x$.

More generally, given two elements $x,y\in B$ with norms in $k^\times$ each of which is orthogonal to the algebra $A_{i}$ generated by the other, we shall denote by $A_{x,y} = k\langle x,y\rangle$ the $k$-subalgebra generated by $x$ and $y$.

Finally, given three elements $x,y,z \in B$ each of which has norm in $k^\times$ and is orthogonal to the algebra $A_{i,j}$ generated by the other two, we shall denote by $A_{x,y,z} = k\langle x,y,z\rangle$ the $k$-subalgebra generated by $x$, $y$ and $z$.
\end{nota}

The following propositions give some motivation to the choice of letter $A$ to denote these algebras, in particular these shall all be CD-subalgebras of $B$ which can be used to inductively define $B$ as a CD-algebra.

\begin{prop}
Let $B$ be a CD-algebra of rank at most $8$ over $k$, and let $x\in B$ be orthogonal to $1\in B$ with $N_B(x)\in k^\times$.
Then the subalgebra $A_x$ generated by $x$ is a quadratic algebra over $k$ with $\sigma_B$ the involution on $B$ restricting to the non-trivial involution of $A_x$.
\end{prop}
Follows from \cite[Thm. 33.17]{book_of_involutions}.
\begin{proof}[Sketch of Proof]
Since $B$ is alternative we know that $k\langle x \rangle$ is associative and commutative.
The condition that $(x,1_B) = 0$ implies that $\sigma(x) = -x$. It follows that:
 \[ x^2 = -N_B(x) \in k. \]
\end{proof}
The above argument shows in fact that all rank $2$ subalgebras are commutative.

\begin{prop}
Let $B$ be a CD-algebra of rank $4$ or $8$ over $k$, and let $x,y\in  B$ have norms in $k^\times$. Suppose that $x$ is orthogonal to $1$ and $y$ is orthogonal to the subalgebra $A_x$.
Then the subalgebra $A_{x,y}$ generated by $x$ and $y$ is a quaternion algebra. Moreover, $\sigma_B$, the involution on $B$, restricts to the standard involution of $A_{x,y}$.
The algebra $A_{x,y}$ is obtained via the Cayley-Dickson construction with $A_x$ and $\delta=-N_B(y)$.
\end{prop}
Follows from \cite[Thm. 33.17]{book_of_involutions}.
\begin{proof}[Sketch of Proof]
It follows from the previous proposition that both $x^2$ and $y^2$ are elements of $k$ and, moreover, that $\sigma(x)=-x$ and $\sigma(y)=-y$.
The condition $(x,y) = 0$ implies that $x\sigma(y) + y\sigma(x) = 0$.
From this it follows that $xy= -yx$.
Since $B$ is alternative we can use this relation to show that $k\langle x,y \rangle$ is associative, consequently, the algebra is a quaternion algebra.
\end{proof}
The above argument shows in fact that all rank $4$ subalgebras are associative.

\begin{prop}
Let $B$ be a CD-algebra of rank at most $8$. The norm $N_B$ is multiplicative and consequently does not depend on parenthesis or rearrangement of terms.
\end{prop}
See \cite[Thm. 33.17]{book_of_involutions}.
\begin{proof}
It suffices to show that $N_B(xy) = N_B(x)N_B(y)$. 
This expression is being evaluated in the rank $4$ subalgebra generated by $x,y$. This is an associative algebra, and the result is thus clear.
\end{proof}

\begin{rmk}
We have shown that the following hold:
\begin{enumerate}
\item
If $A=k$ and $\sigma_A$ is trivial then $B$ is a quadratic extension of $k$. 
\item
If $A$ is a quadratic algebra over $k$ with $\sigma_A$ the unique non-trivial involution, then $B$ is a quaternion algebra.
\item
If $A$ is a quaternion algebra over $k$ with $\sigma_A$ the standard involution, then $B$ is an octonion algebra.
\end{enumerate}
\end{rmk}

\begin{prop}
Let $B$ be a CD-algebra of rank $8$ over $k$, let $x\in  B$ be orthogonal to $1$, let $y\in B$ be orthogonal to the subalgebra $A_x$ generated by $x$ and let $z\in B$ be orthogonal to the subalgebra $A_{x,y}$ generated by $x$ and $y$. Suppose further that all of $x,y,z$ have norms in $k^\times$.
Then the elements $x,y,z$ generate $B$ over $k$ and $B$ is isomorphic to the algebra obtained by the Cayley-Dickson construction with $A_{x,y}$ and $\delta=-N_B(z)$.
\end{prop}
Follows from proof of \cite[Thm. 33.17]{book_of_involutions}.
\begin{proof}[Sketch of Proof]
One can most easily confirm the multiplication law by using the fact that each of $A_{x,z}$, $A_{y,z}$ and $A_{xy,z}$ is a quaternion algebra.

Moreover, the multiplicitivity of the norm allows one to check that $\{ 1,x,y,xy,z, xz,yx, x(yz) \}$ are an orthogonal basis.
\end{proof}

\begin{rmk}
The multiplication rules for $\{ 1,x,y,xy,z, xz,yx, xyz \}$ are entirely described by the fact that the pairs $\{x,y\}$, $\{x,z\}$, $\{y,z\}$ generate quaternion algebras, that $(xy)z = -x(yz)$ and that $z,xz,yx,x(yz)$ are interchangeable in the previous statement.
\end{rmk}

\subsection{Auxilliary Structures on CD-Algebras}

In this section we shall construct some auxilliary structures on CD-algebras.
These shall all be constructed using the algebra operations of the underlying CD-algebra.
Consequently, automorphism of the algebra which preserve the data defining the structure, must also preserve the structure.
These structures will thus allow us to study the automorphism group of the algebra.

\begin{prop}
Let $B$ be a CD-algebra of rank at most $8$ over $k$, and let $x\in  B$ be orthogonal to $1$  with $N_B(x) \in k^\times$.
Let $A_x$ be the subalgera generated by $x$ and $A_x^\perp$ be the $k$-submodule of $B$ consisting of elements perpendicular to $A_x$.
Then $A_x^\perp$ and $B$ have the canonical structure of an $A_x$-module.
\end{prop}
\begin{proof}
We first show that $B$ has an $A_x$-module structure.
As left multiplication by elements of $A_x$ is linear on $B$, the main concern is that for $a_1,a_2\in A_x$ and $y\in B$ we need to show:
\[ (a_1a_2)y = a_1(a_2y). \]
We need only check this on a basis of $B$, thus can assume $N_B(y)\in k^\times$. In this case all of these operations take place in the algebra $A_{x,y}$, which is a quaternion algebra, and hence associative.

It now only remains to show that this action of $A_x$ descends to an action of $A_x^\perp$.
All elements of $A_x$ are of the form $a+bx$, such an element acts on $y\in A_x^\perp$ by sending it to $ay+bxy$.
As $ay\in A_x^\perp$, it remains only to show $bxy\in A_x^\perp$, or equivalently, $xy\in A_x^\perp$.
This is true in the quaternion algebra $A_{x,y}$, hence is true in $B$.
\end{proof}

\begin{prop}
Let $B$ be a CD-algebra of rank at most $8$ over $k$, and let $x\in  B$ be orthogonal to $1$ with $N_B(x)\neq 0$.
Let $A_x$ be the subalgera generated by $x$.
Then:
\[ (ay,z) = (y,\sigma(a)z) \]
for all $a\in A_x$ and $y,z\in B$.
Moreover, the bilinear form:
\[ S(y,z) = (xy,z) \]
is an alternating form on $B$.
\end{prop}
\begin{proof}
By the multiplicativity of the norm (and the polarization identity) we have:
\[ (ay,az) = N_B(a)(y,z). \]
It follows that if $N_B(a)\in k^\times$:
\[ (ay,a(a^{-1})z)) = (a\sigma(a))(y,a^{-1}z) = (y,\sigma(a)z). \]
Since it suffices to check this identity for $a=x$, the result follows immediately.
\end{proof}

\begin{prop}
Let $B$ be a CD-algebra of rank at most $8$ over $k$, and let $x\in  B$ be orthogonal to $1$.
Let $A_x$ be the subalgera generated by $x$.

Then $A_x^\perp$ has the (canonical up to rescaling) structure of an $A_x$-Hermitian space by applying the polarization identity to the $k$-valued form $N_B|_{A_x^\perp}$.
More concretely the form:
\[ H(y,z) = \tfrac{1}{2}(y,z) + \tfrac{1}{2x}(xy,z) \]
is an $A_x$-valued Hermitian form on $A_x^\perp$.
\end{prop}
\begin{proof}
It is clear that the form $H$ is $k$-linear.
We must check $A_x$-linearity, and conjugate symmetry. 
For conjugate symmetry notice:
\begin{align*}
H(y,z) &= \tfrac{1}{2}(y,z) + \tfrac{1}{2x}(xy,z) \\
           &= \tfrac{1}{2}(z,y) + \tfrac{1}{2x}(y,-xz) \\
           &= \tfrac{1}{2}(z,y) + \tfrac{1}{2x}(-xz,y) \\
           &=  \tfrac{1}{2}(z,y) - \tfrac{1}{2x}(xz,y) \\
           &= \sigma(H(y,z)).
\end{align*}
For $A_x$-linearity it suffices to check the linearity with respect to the action of $x$. Indeed we compute
\begin{align*}
 H(xy,z) &= \tfrac{1}{2}(xy,z) + \tfrac{1}{2x}(x^2y,z)  \\
             &= \tfrac{1}{2}(xy,z) + \tfrac{x^2}{2x}(y,z)\\
             &= x(\tfrac{1}{2x}(xy,z) + \tfrac{1}{2}(y,z))\\
             &=  x H(y,z).
\end{align*}
\end{proof}

\begin{prop}
Let $B$ be a CD-algebra of rank at most $8$ over $k$, let $x\in  B$ be orthogonal to $1$, and let $A_x$ be the subalgera generated by $x$.
There is a map $\Pi : A_x^\perp \times A_x^\perp \rightarrow A_x^\perp$
given by 
\[ (a,b) \mapsto \tfrac{1}{2}(ab-ba - H(ab-ba,1_B)). \]
The map $\Pi$ is alternating and $\sigma_{A_x}$-linear.
\end{prop}
\begin{proof}
The map is manifestly alternating and $k$-linear. We may thus work with a basis for $A_x^\perp$.
Fix $y,z$ as usual.
By the symmetry of the setup it suffices to consider the cases $(a,b) = (y,y), (y,xy),(y,z)$.
In the first two cases $\Pi(a,b)$ is easily checked to be zero as the computations all taking place in the quaternion algebra $A_{x,y}$.
In the third case $\Pi(y,z) = yz$. We recall that $(xy)z = -x(yz)$ and conclude that $\Pi$ is $\sigma_{A_x}$-linear.
\end{proof}

\begin{rmk}
The above map is analogous to the cross product.
\end{rmk}

\begin{prop} 
Let $B$ be a CD-algebra of rank at most $8$ over $k$, let $x\in  B$ be orthogonal to $1$, and
let $A_x$ be the subalgera generated by $x$.

The $A_x$-valued $k$-trilinear map:
\[ T(a,b,c) = \tfrac{1}{2}H(c, ab - ba) = H(c,\Pi(a,b)) \]
on $A_x^\perp$ is a non-trivial $A_x$-linear alternating map.
\end{prop}
\begin{proof}
The formula for $T$ in terms of $\Pi$ follows from the observation that $H(ab-ba,1_B) \in A_x$ is perpendicular to $c\in A_x^\perp$.
The $A_x$-linearity is then clear in light of the previous proposition.
The map is also clearly alternating with respect to $a,b$, so it suffices to check with respect to $a,c$.
Thus it suffices to check:
\[ T(a,b,a) = 0 \]
for all $a,b\in A_x^\perp$.
We may do this with respect to an $A_x$-basis $\{y,z,yz\}$.
We already know this holds for cases where $a=b$.
Moreover, as the setup is algebraically symmetric with respect to any choice of two non-equal basis elements, it suffices to take $a=y$ and $b=z$.
\begin{align*}
T(y,z,y)  &= H(y, yz-zy )\\
               &= H(y,2yz) \\
               &= H(\sigma(y)y, 2z) \\
               &= 0.
\end{align*} 
\end{proof}

\begin{rmk}
The trilinear form $T$ is essentially giving the $A_x$-module determinant on $A_x^\perp$.
\medskip

It is important to note that the $A_x$-module structure, the Hermitian structure, the alternating form and the trilinear map are defined only using the CD-structure on $B$ and the subalgebra $A_x$.

\medskip
It is also important to note that the $A_x$-module structure maps are not automorphisms of $B$.

\medskip
Finally, it is worth noting that the Hermitian form and bilinear form satisfy:
\[ (y,z) = \Tr_{A_x/k}(H(y,z)). \]
\end{rmk}

\subsection{Classifications of CD-algebras}

In this section we shall classify the CD-algebras of ranks $2$, $4$ and $8$ over $p$-adic fields and number fields.

\begin{prop}
The isomorphism class of a CD-algebra $B$ of rank at most $8$ is determined by its rank and its isomorphism class as a quadratic space.
\end{prop}
See \cite[Thm. 33.19]{book_of_involutions}.
\begin{proof}[Sketch of proof]
We proceed somewhat inductively.
The isomorphism class of a rank $2$ CD-algebra $A_x$ is determined by $x^2 = -x\sigma(x) = N_{A_x}(x) \in k^\times$.
The value for $N_{A_x}(x)$ must be represented by a vector perpendicular to an element of norm $1$.

The isomorphism class of a rank $4$ CD-algebra $A_{x,y}$ is determined by a choice of subalgebra $A_x$ and an element $y\in A_x^\perp$.
The options for the isomorphism class of $A_x$ come from values represented by the form $N_{A_x}(x)$ on a subspace perpendicular to vector of norm $1$. (The isomorphism class of such a quadratic space is well-defined.)
The choice of values for $y^2$ comes from those values represented in $A_x^\perp$. By construction, $A_x^\perp \simeq \delta A_x$ as a quadratic space, and thus the choice of $\delta$ is well-defined up to norms of $A_x$ once $A_x$ is chosen.
It is well-known a quaternion algebra is determined by a choice $A_x$ and an element of $\delta \in k^\times/N_{A_x/k}(A_x^\times)$.

For CD-algebras of rank $8$, as above we will find $A_{x,y}^\perp \simeq \delta A_{x,y}$. Thus $\delta$ is determined up to norms from $A_{x,y}$ by the choice of subalgebra $A_{x,y}$ and the quadratic space. By construction an octonion algebra is determined by $A_{x,y}$ and an element of $k^\times/N_{A_{x,y}/k}(A_{x,y}^\times)$. Note that, as with quaternion algebras, non-isomorphic $A_{x,y}$ may give isomorphic $A_{x,y,z}$ under the appropriate choice of $\delta$.

To complete the proof observe that two CD-algebras are isomorphic if and only if they can both be constructed from the same pair $A$ and $\delta$. The options for $A$ are determined by the quadratic form, and $\delta$ is determined by the quadratic form once $A$ is chosen.
\end{proof}
\begin{rmk}
Not all quadratic spaces of a given rank can arise from a CD-algebra.
\end{rmk}

The following facts are reasonably well-known consequences of the above.
\begin{prop}
There is exactly one isomorphism class of CD-algebra over $\bC$ of any given rank.
\end{prop}
There is only one quadratic space over an algebraically closed field.

\begin{prop}
There are exactly two isomorphism classes of CD-algebras over $\bR$ for each rank $2,4$ and $8$.
\end{prop}
The CD-construction ensures the quadratic spaces have signatures $(n/2,n/2)$ or $(n,0)$.

\begin{prop}
Let $k$ be  a local field other than $\bR$ or $\bC$ and of residue characteristic not $2$.
Then there are precisely $4$ isomorphism classes of CD-algebras of rank $2$ over $k$.
\end{prop}
As the quadratic space represents $1$, the Hasse invariant is trivial and the quadratic space is determined by the discriminant.
The structure of the unit group modulo squares is well-known.

\begin{prop}
Let $k$ be a local field of residue characteristic $2$, with ramification degree $e$ and inertial degree $f$ over the prime subfield.
There are precisely $2^{1+e+f}$ isomorphism classes of CD-algebras of rank $2$ over $k$.
\end{prop}
The argument is the same as the previous proposition.

\begin{prop} 
Let $k$ be  a local field other than $\bR$ or $\bC$.
Then there are precisely $2$ isomorphism classes of CD-algebras of rank $4$ over $k$.
\end{prop}
The CD-construction ensures the discriminant is $1$. Hence the forms are determined by their Hasse invariant.

\begin{prop} 
Let $k$ be  a local field other than $\bR$.
Then there is precisely $1$ isomorphism class of CD-algebra of rank $8$ over $k$.
\end{prop}
The CD-construction ensures the discriminant and Hasse invariant are trivial. There is thus only one such form.

\begin{prop}
Let $k$ be a global field. Let $s$ be the number of real embeddings of $k$.
There are precisely $2^s$-many isomorphism classes of CD-algebras of each rank $8$ over $k$.
\end{prop}

This completes our classification of forms of CD-algebras.

}

\section{The Groups of type G2}
\label{sec:g2}

For the remainder of this section fix $\bO$ an octonion algebra over $k$. We shall denote by $\bO^0$ the traceless elements of $\bO$, that is, those perpendicular to $1$.
In this section we shall study the group scheme:
\[ G(K) = \Aut_K(K\otimes_k \bO). \]
Such a $G$ will be a semi-simple algebraic group of type G2. Moreover, all groups of type G2 arise this way for different choices of $\bO$ (see Proposition \ref{prop:FormsG2}).

\begin{conv}
Whenever we write $A_{x}$, $A_{x,y}$ or $A_{x,y,z}$ we shall assume that all of $x,y,z \in \bO^0$ have norms in $k^\times$ and that each of $x,y$ and $z$ are orthogonal to the algebra generated by the other two.
\end{conv}

\subsection{Some Elements of $G$}

Before describing the structure of the group $G$, we shall construct some special elements of $G$.

\begin{prop}
Any automorphism $\varphi$ of $A_{x,y}$ can be extended to an automorphism $\tilde\varphi$ of $A_{x,y,z}$ acting trivially on $z$.
The map takes $[a,b] \mapsto [\varphi(a) , \varphi(b)]$.
\end{prop}
\begin{proof}[Sketch of proof]
As the multiplication in $A_{x,y,z}$ is defined $k$-linearly in terms of that of $A_{x,y}$ there is no obstruction and one easily checks that $\varphi$ preserves both $\sigma$ and multiplication on $A_{x,y,z}$.
\end{proof}

\begin{nota}
We shall denote by $A_x^1$ the elements of $A_x$ of norm $1$.

We shall denote by $T_{A_x,\sigma}$ the associated algebraic torus whose functor of points is:
\[ T_{A_x,\sigma}(R) = \{ y \in A_x\otimes_k R \mid y\sigma(y) = 1 \}. \]
\end{nota}

\begin{prop}
Let $a\in A_x$ be an element of norm $1$.
The map $A_{x,y} \rightarrow A_{x,y}$ defined by:
\[ 1\mapsto 1\quad x\mapsto x \quad y\mapsto ay \]
can be extended to an automorphism $\varphi_{a,y}$ of $A_{x,y}$.
The map $\varphi_{a,y}$ acts trivially on $A_x$ and as multiplication on the left by $a$ on the $k$-span of $y$ and $xy$.
\end{prop}
\begin{proof}[Sketch of proof]
The element $ay$ has the same norm as $y$, hence $A_{x,y}\simeq A_{x,ay}$.
\end{proof}

\begin{prop}
Denote by $A_x^1$ the elements of $A_x$ of norm $1$.
For any choice of $a,b\in A_x^1$
the automorphisms $\tilde\varphi_{a,y}$ and  $\tilde\varphi_{b,z}$ of $A_{x,y,z}$ commute.
\end{prop}
\begin{proof}[Sketch of proof]
We have a decomposition:
 \[ A_{x,y,z} = \langle 1,x \rangle \oplus \langle y,xy \rangle \oplus \langle z,xz \rangle \oplus \langle yz,x(yz) \rangle. \]
We identify each of the four factors with $A_x$ in the obvious way.
The map $\tilde\varphi_{a,y}$ acts trivially on the first and third factor, as multiplication by $a$ on the second and multiplication by $\sigma(a)$ on the fourth.
The map $\tilde\varphi_{b,z}$ acts trivially on the first and second factor, as multiplication by $b$ on the third and multiplication by $\sigma(b)$ on the fourth.
These actions are clearly compatible and hence commute.
\end{proof}

\begin{prop}
\label{prop:Phi}
The map $\Phi : A_x^1\times A_x^1 \rightarrow \Aut(\bO)$ given by:
\[ (a,b) \mapsto \tilde\varphi_{a,y}\tilde\varphi_{b,z} \]
is an injective homomorphism.
\end{prop}
\begin{proof}[Sketch of proof]
We can recover the values of $a$ and $b$ from the action on the second and third factors, respectively.
\end{proof}

\begin{rmk}
The above map induces a map $\Phi$ from the algebraic torus $(T_{A_x,\sigma})^2$ into $G$.
\end{rmk}

\begin{prop}
Fix $a\in A_x^1$.
Left multiplication by $a$ induces a map $m_a: A_{x}^\perp \rightarrow A_{x}^\perp$ which commutes with the image of $\Phi$.
The map $m_a$ induces a map $T_{A_x,\sigma} \rightarrow \SO_{\bO^0}$, however, it is only an automorphism of $\bO$ if $a^3=1$.
Moreover, the image of $\Phi$ is its own centralizer in $\Aut(\bO)$.
\end{prop}
\begin{proof}[Sketch of proof]
The commutativity is clear given the compatibility of the actions; one easily checks that it preserves the norm form.

First, notice that the centralizer of $\Phi$ in $\SO_{\bO^0}$ is generated by the image of $\Phi$ and $m_a$. The claim about centralizers now follows by observing that whenever $a^3=1$ the image of $m_a$ is in the image of $\Phi$.
\end{proof}

The following proposition is an immediate consequence of the above.
\begin{prop}\label{prop:amaxtori}
The torus $(T_{A_x,\sigma})^2$ embedded into $G$ via $\Phi$ is a maximal torus in $G$.
\end{prop}

\subsection{Describing Automorphisms of $\bO$}

In this section we give several methods of describing automorphisms of $\bO$ which shall be of use in describing the embeddings of groups of type A2 into those of type G2.

We now give several methods of describing or specifying an element of $G$.
\begin{enumerate}
\item
Fix an isomorphism $\bO \simeq A_{x,y,z}$.

The automorphism group of $\bO$ is in bijection with the set of triples $(x',y',z')$ in $\bO^0$ with $N(x) = N(x')$,  $N(y) = N(y')$,  $N(z) = N(z')$, $x'$ perpendicular to $y'$, and $z'$ perpendicular to the algebra generated by $x'$ and $y'$.
Given such a triple we obtain the automorphism which takes $x\mapsto x'$, $y\mapsto y'$ and $z\mapsto z'$.
By the Cayley-Dickson construction we have isomorphisms $A_{x'}\simeq A_{x}$, $A_{x',y'}\simeq A_{x,y}$ and $A_{x',y',z'} \simeq A_{x,y,z}$.
It is clear that all automorphisms arise uniquely this way.

\item
The following description is an easy consequence of the first.
As automorphisms of $\bO$ must preserve both the identity element of $\bO$ and the norm on $\bO$ we find that conclude that $G$ has an embedding:
\[ G \injects \SO_{\bO^0} \injects \GL(\bO^0). \]
As $\bO = A_{x,y,z}$ is generated by $x,y$ and $z$, an automorphism of $\bO$ is determined by where it sends these generators.

The condition that $G\subset \SO_{\bO^0}$ implies that these will be taken to orthogonal elements of the same length.
It is clear we shall need the conditions $g(xy)=g(x)g(y)$, $g(xz)=g(x)g(z)$, $g(yz)=g(y)g(z)$ and $g((xy)z) = g(xy)g(z)$. The multiplicativity of the norm means that these conditions do not actually result in restrictions on the possible images of $x$, $y$ and $z$.
The only additional condition that describes $G$ as a subgroup of $\SO_{\bO^0}$ is that $g(z)$ must be perpendicular to $g(x)g(y)$.
As once this is satisfied, the triple $(g(x),g(y),g(z))$ satisfies the conditions of the first description.
In particular:
\[  G = \{ g\in \SO_{\bO^0} \mid g(xy)=g(x)g(y), \cdots, g((xy)z) = g(xy)g(z)\text{ and } (g(z), g(x)g(y)) = 0 \}. \]

\item
The number of additional equations in the previous description makes it unwieldy to work with directly.
An alternative description is as follows.

Every element of $\SO_{\bO^0}$, and hence of $G$ stabilizes a non-isotropic $k$-linear subspace of $\bO^0$, and consequently, also its orthogonal complement.
We may select a $k$-rational basis element $x$ of this non-isotropic space.
Thus every element of $G$ is determined by a choice of $x$ it stabilizes and an element of $\Stab_x(G) \injects \SO_{A_x^\perp}$.
Note that the choice of $x$ may not be unique, though it is generically unique.

Fix two elements $y,z$ of $A_x^\perp$ which are $A_x$-perpendicular (that is, perpendicular with respect to the $A_x$-Hermitian form on $A_x^\perp$).
Any element of $G$ which stabilizes $x$ must preserve the  $A_x$-Hermitian pairing $H$ on $A_x^\perp$.
Thus we conclude that $\Stab_x(G) \subset \UU_{A_x^\perp}$. 
Note that preserving $H$ implies that the images of $y,z$ have the appropriate lengths, that $g(xy)=g(x)g(y)$ and that $(g(z), g(x)g(y)) = 0$.
Furthermore, an element of $G$ must also preserve the $A_x$-trilinear form $T$ on $A_x^\perp$, that is, it must have trivial $A_x$-determinant.
In particular $\Stab_x(G) \subset \SU_{A_x^\perp}$. We claim this inclusion is an equality.

Now, an element of $\Stab_x(G)$ is determined by where it sends $y$ and $z$. The only condition being that $y$ and $z$ must be perpendicular with respect to $H$. The third $A_x$-basis element $yz$ of $A_x^\perp$ must then be sent to $g(y)g(z)$.
By contrast, an element $\UU_{A_x^\perp}$ must send  $y$ and $z$ to elements which are $H$-perpendicular, however, the condition on where it sends $yz$ is only that it must be $H$ perpendicular to both $y$ and $z$. This space is one-dimensional (over $A_x$), thus the choice is only determined up to $A_x^1$. The additional condition that $T$ be preserved removes this ambiguity and implies that, for an element of $\SU_ {A_x^\perp}$ the image of $yz$ is determined by the images of $y$ and $z$.
We claim these conditions imply that $yz$ must be sent to $g(y)g(z)$. Indeed we have:
\[ H(g(yz),g(yz))=H(yz,yz)=T(y,z,yz) = T(g(y),g(z),g(yz)) = H(g(y)g(z),g(yz)) \]
which implies that for $g\in \SU_{A_x^\perp}$ we will have $g(yz)=g(y)g(z)$.

We conclude that we may describe $G$ as:
\[ G(R) = \underset{x\in\bO_R^\times/R^\times}\cup  \SU_{A_x^\perp}(R) . \]
Moreover we have maps:
\[  \Stab_x{G} \injects G \surjects Gx \]
with the orbit $Gx$ of $x$ being six dimensional and $\SU_{A_x^\perp}$ being eight dimensional, thus $G$ is $14$ dimensional.
\end{enumerate}

The discussion above allows us to conclude the following:

\begin{prop}
The group $G(k)$ acts transitively on elements of $\bO^0$ of a given norm.
\end{prop}

\begin{prop}
Two subgroups $\SU_{A_x^\perp}$ and $\SU_{A_{y}^\perp}$ of $G$ are conjugate and isomorphic if and only if $A_x \simeq A_{y}$.
\end{prop}

\begin{prop}
The group $G$ has rank $2$ and dimension $14$.
\end{prop}

\begin{rmk}
We have already exhibited some maximal tori of $G$, namely the images of the maps $\Phi$. Not all maximal tori arise this way.
\end{rmk}

\section{Forms of Algebraic Groups}
\label{sec:forms}

In order to fully describe triples $(H_k, G_k, H_k\rightarrow G_k)$:
which after base change become isomorphic to our fixed triple:
$(H_{\overline{k}},G_{\overline{k}},H_{\overline{k}} \rightarrow G_{\overline{k}})$
it shall first be necessary to understand the forms of $G_k$ and of $H_k$.

In this section we shall briefly recall several standard results concerning the Galois cohomology of algebraic groups which allow for a concrete descriptions of the groups we wish to study. Most of the results of this section are given with little or no proof. A more thorough and rigourous treatment of this material can be found in \cite{book_of_involutions}.

\begin{prop}\label{prop:formsofgroups}
Fix an algebraic group $G_k$ over $k$.
The isomorphism classes of groups $\tilde{G}_k$ over $k$ with $\tilde{G}_{\overline{k}} \simeq G_{\overline{k}}$ is in bijection with:
\[ H^1( \Gal(\overline{k}/k), \Aut(G_k(\overline{k})) ). \]
\end{prop}
This is a standard result of Galois descent, see \cite[Cor III.1.3]{Serre_cohom}.

In order to compute $H^1( \Gal(\overline{k}/k), \Aut(G_k(\overline{k})) )$ we observe that we have exact sequences:
\[  1 \rightarrow Z(G_k) \rightarrow G_k \rightarrow G_k^{\rm adj} \rightarrow 1 \]
and 
\[ 1 \rightarrow G_k^{\rm adj} \rightarrow \Aut(G_k) \rightarrow {\rm Out}(G_k) \rightarrow 1 \] 
where ${\rm Out}(G_k) $ denotes the outer automorphisms and $G_k^{\rm adj}$ denotes the adjoint form of $G_k$.

\subsection{Forms of G2}

It is well-known that for a semi-simple group ${\rm Out}(G_k)$ corresponds to automorphisms of the Dynkin diagram.
For G2 this group is trivial. As the center $Z(G_k)$ is also trivial,
it follows that:
\[ H^1( \Gal(\overline{k}/k), \Aut(G_k(\overline{k})) ) = H^1( \Gal(\overline{k}/k), G_k(\overline{k}) ) = H^1(\Gal(\overline{k}/k), \Aut(\bO_{\overline{k}})). \]

\begin{prop}\label{prop:FormsG2}
The cohomology group:
\[  H^1(\Gal(\overline{k}/k), \Aut(\bO_{\overline{k}})) \]
is in bijection with isomorphism classes of algebras $\tilde{\bO}_k$ such that $\tilde{\bO}_{\overline{k}} \simeq \bO_{\overline{k}}$.

Moreover, the forms of $G_k$ are all of the form $\Aut(\tilde{\bO})$ where $\tilde{\bO}$ is a form of $\bO$.
\end{prop}
See \cite[Thm. 26.19]{book_of_involutions}.

\begin{rmk}
Though we have shown already that all of the forms of $\bO$ come from the real places of $k$, this can now also be seen as a consequence of the fact that $G$ is simply connected, as the local cohomology of simply connected groups is trivial for local fields other than $\bR$. Hence,
\[  H^1( \Gal(\overline{k}/k), G_k(\overline{k}) ) \injects \prod_\nu  H^1( \Gal(\overline{k_\nu}/k_\nu), G_k(\overline{k_\nu}) ) = \prod_{\nu || \infty}  H^1( \Gal(\overline{k_\nu}/k_\nu), G_k(\overline{k_\nu}) ). \]
\end{rmk}

\subsection{Forms of A2}

We sketch the classification of forms of simply connected groups of type A2, more thorough treatements of this material with complete proofs can be found in \cite{book_of_involutions}.

Simply connected groups $H$ of type A2 have an outer automorphism group of order $2$ and a center of order $3$ which we shall denote $\mu_3^\xi$. Note that $\mu_3^\xi$ is a twisted form of the group of cube roots of unity.

From the exact sequence:
\[ H^1( \Gal(\overline{k}/k), H_k^{\rm adj}(\overline{k}) ) \rightarrow H^1( \Gal(\overline{k}/k), \Aut(H_{\overline{k}}) ) \rightarrow H^1(\Gal(\overline{k}/k), \{\pm1\} ). \]
We deduce that associated to a form of $H$ is an element of $H^1(\Gal(\overline{k}/k), \{\pm1\} )$.
This determines a quadratic \'etale algebra $A$ over $k$.

From the exact sequences:
\[ H^1( \Gal(\overline{k}/k), H_k(\overline{k}) ) \rightarrow H^1( \Gal(\overline{k}/k), H_k^{\rm adj}(\overline{k}) ) \rightarrow H^2(\Gal(\overline{k}/k), \mu_3^\xi )  \]
we associate to a form of $H$ an element of $H^2(\Gal(\overline{k}/k), \mu_3^\xi )$.

The group $\mu_3^\xi$ arises from the following exact sequence:
\[ 1 \rightarrow \mu_3^\xi \rightarrow T_{A,\sigma} \overset{[3]}\rightarrow T_{A,\sigma} \rightarrow 1. \]
The group $ T_{A,\sigma}$ denoting the torus of norm one elements of $A$ over $k$.
We obtain the following exact sequence:
\[ H^1(\Gal(\overline{k}/k),  T_{A,\sigma} ) \overset{[3]} \rightarrow H^1(\Gal(\overline{k}/k),  T_{A,\sigma} ) \rightarrow H^2(\Gal(\overline{k}/k),\mu_3^\xi ) \rightarrow H^2(\Gal(\overline{k}/k), T_{A,\sigma} ) \]
The group $H^2(\Gal(\overline{k}/k), T_{A,\sigma} )$ classifies the kernel of the corestriction map from the Brauer group of $A$ to that of $K$.
The group $H^1(\Gal(\overline{k}/k),  T_{A,\sigma} )$ is isomorphic to $k^\times/N_{A/k}(A^\times)$.

Extending the exact sequence we find we may associate to $H$ a degree $3$ central simple algebra $M$ over $A$ such that $M\otimes_{A,\sigma} A \simeq M^{\rm op}$.
These are precisely the central simple algebras $M$ with involutions $\tau$ restricting to the non-trivial involution of $A$.

Inspecting the exact sequence further we observe that the map $[3]$ surjects onto $k^\times/N_{A/k}(A^\times)$ and thus we only obtain the trivial class in this cohomology group.

Finally, to a form of $H$ we may associate an element of $H^1( \Gal(\overline{k}/k), H_k(\overline{k}) ) $.
As $H$ is simply connected, this cohomology is supported at the real places $\nu$ of $k$.
If $A_\nu \simeq k_\nu \times k_\nu$ then $H_{k_\nu}$ is $\SL_3$ as there are no degree three central simple algebras over $\bR$. In this case $H^1( \Gal(\overline{k}/k), H_k(\overline{k}) ) $ is trivial.
If $A_\nu \simeq \bC$ then $H$ is the special unitary group for a Hermitian form on a $3$ dimensional space. In this case $H^1( \Gal(\overline{k}/k), H_k(\overline{k}) ) $  classifies the Hermitian forms with the same discriminant as the one defining $H$. There are precisely two, a definite form and an indefinite form.

\begin{rmk}
Note that when using the above exact sequences to describe $H^1( \Gal(\overline{k}/k), \Aut(H_{\overline{k}}) )$ one must twist by elements of $H^1(\Gal(\overline{k}/k), \{\pm1\} )$ before considering the kernel $H^1( \Gal(\overline{k}/k), H_k^{\rm adj}(\overline{k}) ) $. The effect is that one picks the \'etale agebra, $A$ then the algebra $M$ and then the definiteness of the form at real places which ramify in $A$.
\end{rmk}

\begin{cons*}
Let $A$ be a rank $2$-\'etale algebra over $k$.
Let $M$ be a degree $3$ simple algebra over $A$, with an involution $\tau$ restricting to the involution $\sigma$ on $A$.
Let $J$ be an element of $M$ such that $\tau(J)=J$.

Define $\tilde{H}_k$ to be the algebraic group whose functor of points is:
\[ \tilde{H}_k(R) = \{ g\in (M\otimes_k R)^\times \mid \tau(g)Jg = J \}. \]
It is an easy exercise to check that $\tilde{H}$ is isomorphic to $H$ over $\overline{k}$.
\end{cons*}

\begin{rmk}
Note that the choice of $J$ is equivalent to the choice of $\tau$ (see \cite[Prop. 2.18]{book_of_involutions}).

At real places of $k$ which ramify in $A$ we may suppose that $\tau$ is the conjugate transpose, in which case the definiteness or indefiniteness of $\tilde{H}_k$  is controlled by whether the eigenvalues of $J$ all have the same sign, or different signs at $\nu$.
\end{rmk}

\begin{prop}\label{prop:FormsA2}
The forms of $H_k$ are precisely the groups $\tilde{H}_k$ as constructed above. 
\end{prop}
See \cite[Thm. 26.9]{book_of_involutions}.

It follows that given a form $\tilde{H}$ of $H_k$, the algebras $A$ and $M$ are uniquely determined, as is the definiteness/indefiniteness of $J$.

\subsection{Forms of $\bG_m^2$}

In this section we sketch the basic Galois cohomological concepts related to the classification of tori, more thorough treatments can be found in \cite{ReederElliptic} or \cite{WalkerThesisTori}.

It is well-known that the automorphisms of tori come from automorphisms of the character lattice, from this we conclude:
\begin{prop}
The forms of $\bG_m^2$ are in bijection with:
\[ H^1(\Gal(\overline{k}/k), \GL_2(\bZ)). \]
\end{prop}
This is the specialization of Proposition \ref{prop:formsofgroups} to this context.

Computing $H^1(\Gal(\overline{k}/k), \GL_2(\bZ))$ is, however, unnecessary for our applications, as we are not interested in all forms of $\bG_m^2$, only those that appear in $G$.
We give only sketches of the following, as they will not be needed in the sequel.

\begin{prop}\label{prop:ToriInGroups1}
Fix a split semi-simple algebraic group $G$ of rank $r$, a split maximal torus $T$, with Weyl group $W=N_G(T)/T$. Denote by $O$ the outer automorphism group of $G$ whose action on $G$ stabilizes $T$.
The forms of $T$ appearing in forms of $G$ are a subset of: 
\[ H^1(\Gal(\overline{k}/k), O\ltimes W) \injects H^1(\Gal(\overline{k}/k),\Aut(T)).\]
The map arising from the action of $O\ltimes W$ on the root lattice of $G$.
\end{prop}
\begin{proof}[Sketch of Proof]
We remark that we may actually select a representative section for $O$ acting as automorphisms of $G$ so that $O$ preserves $T$.
Denote by $N$ the normalizer in $G$ of $T$ so that $W=N/T$.
We obtain maps $O\ltimes W \rightarrow \Aut(T)$ and $O\ltimes N \rightarrow \Aut(G)$ and thus maps:
\[ H^1(\Gal(\overline{k}/k), O\ltimes W) \injects H^1(\Gal(\overline{k}/k),\Aut(T)) \]
and
\[ H^1(\Gal(\overline{k}/k), O\ltimes N) \surjects  H^1(\Gal(\overline{k}/k), \Aut(G)). \]
The surjectivity of the second map is deduced from \cite[Lem 6.10]{PlatinovRapinchuk}.

An element in the image of $H^1(\Gal(\overline{k}/k), O\ltimes N)$ now twists $T$ in two ways.
Firstly via the map to $H^1(\Gal(\overline{k}/k),\Aut(T))$ and secondly via its restriction to $T$ in $H^1(\Gal(\overline{k}/k), \Aut(G))$.
One checks that these give the same twisting.

Finally, that all tori in forms of $G$ arise this way follows by the observation that if we select any other torus $\tilde{T} \subset \tilde{G}$ then there is an $\overline{k}$ map $\phi: G_{\overline{k}} \rightarrow \tilde{G}_{\overline{k}}$ taking $T$ to $\tilde{T}$. One checks that we may use $\phi\gamma(\phi^{-1})$ to obtain the desired element of $H^1(\Gal(\overline{k}/k), O\ltimes N)$.
\end{proof}

\begin{prop}\label{prop:ToriInGroups2}
Fix a split semi-simple algebraic group $G$, a split maximal torus $T$ with Weyl group $W$. Denote by $O$ the outer automorphism group of $G$ whose action on $G$ stabilizes $T$.
Fix $\Phi$ the collection of weights of $G$ appearing in a faithful $k$-rational representation $\rho$ of $G$.
Denote by $\Sigma_\Phi$ the symmetric group on $\Phi$.
The map $W\times O \injects \Sigma_\Phi$ induces a map:
\[ \varphi : H^1(\Gal(\overline{k}/k), O\ltimes W) \rightarrow H^1(\Gal(\overline{k}/k),  \Sigma_\Phi). \]
The set $H^1(\Gal(\overline{k}/k),  \Sigma_\Phi)$ classifies \'etale algebras of dimension $\abs{\Phi}$.

Finally, if $T\injects G$ is associated to $\xi\in H^1(\Gal(\overline{k}/k), O\ltimes W)$ then $\varphi(\xi)$ is associated to an algebra $E$ with $T\injects T_E = \Res_{E/k}(\bG_m)$.
The subtorus $T \subset T_E$ is defined by equations induced from relations on the characters $\Phi$ of $T_E$ coming from the relations on $\Phi$ viewed as weights of $G$.
\end{prop}
\begin{proof}[Sketch of Proof]
The \'etale algebra $E$ may be taken to be the center of the centralizer of $\rho(T)$.
The idempotents of $E$ are the weights of $\rho$ and thus the Galois action on $\Phi$ defined by the map $W\times O \injects \Sigma_\Phi$ determines the algebra $E$.
By construction we have $T\injects T_E = \Res_{E/k}(\bG_m)$.
\end{proof}

\begin{rmk}
One can modify the previous two results to not depend on split groups and split maximal tori by twisting all the exact sequences.
In so doing one replaces for example $\Sigma_\Phi$ by a twisted form $\Sigma_\Phi^\xi$. The cohomology set $H^1(\Gal(\overline{k}/k),  \Sigma_\Phi^\xi)$ still classifies \'etale algebras over $k$, simply with a new base point (in particular the algebra corresponding to the twisting $\xi$).
\end{rmk}

The classification of triples $T_k \injects G_k$ thus amounts to classifying the algebras $E$, describing the subtori $T$, and classifying the rational conjugacy classes of such maps.

\section{Classifications of certain Subgroups of G2}
\label{sec:classifications}

In this section we shall obtain our main results, that is, we shall classify the $G_k(k)$-conjugacy classes of special unitary groups and maximal tori in groups of type G2.

\subsection{Classification of Special Unitary Groups in G2}

\begin{thm}\label{thm:SUclassification}
Fix an octonion algebra $\bO$ over $k$ and let $G_k$ be the algebraic group associated to its automorphism group and set $\tilde{G}_{\overline{k}} = G_{\overline{k}}$.
Fix $x\in \bO\otimes_k\overline{k}$ perpendicular to $1_{\bO}$.
Fix $\tilde{H}_{\overline{k}} = \Stab_x(G_{\overline{k}})$ to be the stabilizer in $G_{\overline{k}}$ of $x$ and fix its natural inclusion:
\[ \tilde{f} : \tilde{H}_{\overline{k}} \injects \tilde{G}_{\overline{k}}. \]
Then the set of triples  $f : H_k \rightarrow G_k$ considered up to $G_k(k)$-conjugacy, that become $ \tilde{G}_{\overline{k}}(\overline{k})$-conjugate to $\tilde{f} : \tilde{H}_{\overline{k}} \injects \tilde{G}_{\overline{k}}$ after base change are in natural bijection with the isomorphism classes of quadratic \'etale algebras $A$ over $k$ that have an embedding $A\injects \bO$.
\end{thm}
\begin{proof}
It is clear that as the group $\tilde{H}_{\overline{k}}$ is defined as the stabilizer of an element $x$, any $G_{\overline{k}}(\overline{k})$-conjugate must also be the stabilzer of an element $x'\in \bO^0\otimes_k\overline{k}$.
We could equivalently have defined it to be the stabilizer of the line $\overline{k}x'$ spanned by $x'$.
Each $G_{\overline{k}}(\overline{k})$-conjugate of $\tilde{H}_{\overline{k}}$ stabilizes a unique line in $\bO^0$, and the $k$-rationality of the conjugate is equivalent to the $k$-rationality of this line.
It follows that the collection of triples $f : H_k \rightarrow G_k$ we are interested in is in bijection with $G_k(k)$-conjugacy classes of $k$-rational lines which are $\tilde{G}_{\overline{k}}(\overline{k})$-conjugate to the line $\overline{k}x$.

We have seen that the automorphism group of the octonions acts transitively on lines containing an element of any given norm.
In particular $\tilde{G}_{\overline{k}}(\overline{k})$ acts transitively on non-isotropic lines, whereas $G_k(k)$ acts transitively on lines with a $k$-rational point of a fixed norm.
The norm of a $k$-rational point on a non-isotropic $k$-rational line is well defined up to $(k^\times)^2$.
Such a point uniquely determines a quadratic \'etale algebra $A\simeq A_x$ together with an embedding $A\injects \bO$.
Conversely any quadratic \'etale algebra $A$ embedding in $\bO$ determines such a $k$-rational line up to $G_k(k)$-conjugacy.
\end{proof}

\begin{cor}
Fix an octonion algebra $\bO$ over $k$ and let $G_k$ be the algebraic group associated to its automorphism group.
The isomorphism classes of simply connected A2 subgroups embedding in $G_k$ are special unitary groups for quadratic \'etale algebras $A$ which embed in $\bO$.
For each $A$ embedding in $\bO$ there is a unique isomorphisms class of special unitary group over $k$ embedding in $G_k$.
\end{cor}
\begin{proof}
The only new claim here is that there are no other embeddings of A2 subgroups in $G_k$ besides those already being considered.
Indeed, any such embedding must take a maximal torus of the A2 subgroup to that of $G_k$. This maximal torus determines uniquely the underlying representation.
There are only two six dimensional representations of A2. They are isomorphic and are interchanged by the outer automorphism of A2 hence give the same embedded subgroup in G2.

Note that we shall see that $G_k(k)$ contains an element acting as the outer automorphism (Prop \ref{prop:normalizer}), hence these representations are moreover $G_k(k)$-conjugate.
\end{proof}

\begin{cor}\label{cor:SUoverR}
Fix an octonion algebra $\bO$ over $k$.
The isomorphism classes of simply connected A2 subgroups embedding in $G_k=\Aut_k(\bO)$ are controlled by the structure of $\bO$ at the real places of $k$.
In particular, fixing a real place $k_\nu$ of $k$ we have the following cases:
\begin{itemize}
\item If $\bO$ is definite at $k_\nu$ then all $A$ appearing are imaginary extensions of $k_\nu$ (the place $\nu$ ramifies in $A$) and the special unitary group $H_{k_\nu}$ is definite.
\item If $\bO$ is indefinite at $k_\nu$ and $A$ is an imaginary extensions of $k_\nu$ then the special unitary group $H_{k_\nu}$ is indefinite.
\item If $\bO$ is indefinite at $k_\nu$ and $A$ is a real extension of $k_\nu$ (the place $\nu$ splits in $A$), then the group $H_{k_\nu}$ is split, that is, $H_{k_\nu} \simeq \SL_{3,k_\nu}$.
\end{itemize}
\end{cor}
\begin{proof}
The first case follows from the observation that $G_{k_\nu}(k_\nu)$ is compact.
The second case follows from the observation that the Hermitian form which arises must be indefinite.
The third case is a consequence of the fact that this is the unique Hermitian space for a split extension.
\end{proof}

\begin{rmk}
The forms of A2 whose construction involves non-trivial cubic division algebras do not appear.
This can be seen in two ways, firstly they do not have (non-trivial) representations of dimension six, secondly such forms do not arise from Hermitian structures on $A$-vector spaces whereas the forms we construct do.
\end{rmk}

\subsection{Classification of Maximal Tori in G2}

\begin{prop}
Every maximal torus $T_k$ of $G_k$ factors through a unique A2 subgroup $H_k$.
\end{prop}
\begin{proof}
First, we claim that every maximal torus $T_k$ of $G_k$ has in its action on $\bO^0$ a trivial eigenspace spanned by a $k$-rational element $x$. 
The existence of a trivial eigenspace can be checked over $\overline{k}$, that the eigenspace is $k$-rational follows from the fact that the torus is $k$-rational.
We claim that this eigenspace is non-isotropic. This again can be checked over $\overline{k}$ and over the algebraic closure all maximal tori are conjugate, and hence all of these subspaces are conjugate. As there exist maximal tori for which this subspace is non-isotropic (Prop \ref{prop:amaxtori}), it follows that it is non-isotropic for all maximal tori.

It follows that $T_k \injects \Stab_x(G_k)$ embeds into the stabilizer of $x$.
The group $\Stab_x(G_k)$ is precisely one of the groups $H_k$ considered above.
\end{proof}

\begin{thm}\label{thm:ToriInG2}
The $G_k(k)$-conjugacy classes of maximal tori in $G_k$ are in bijection with pairs $(\overline{H},\overline{T})$ consisting of a  $G_k(k)$-conjugacy class $\overline{H}$ of A2 subgroup and an $H(k)$-conjugacy class $\overline{T}$ of maximal torus $T$ in $H$.
\end{thm}
\begin{proof}
Since for each maximal torus $T$ of $G$ the subgroup $H$ containing $T$ is unique, we may partition the $G_k(k)$-conjugacy classes according to the conjugacy class of the group $H$ containing it.

To complete the result  we must now show that distinct $H(k)$-conjugacy classes give distinct $G(k)$-conjugacy classes.
To see this we must show that any element of $G(k)$ which normalizes $H$, takes $T$ to an $H(k)$-conjugate of $T$.
Indeed, the normalizer of $H$ in $G_k$ is generated by a single element $g$ of order $2$ (see Prop \ref{prop:normalizer}) and conjugation by $g$ acts as the outer automorphism of $H$.
The outer automorphism of $H$ comes from the composition of the adjoint map on the Hermitian space and the inverse. As tori in unitary groups are self adjoint, $g$ takes $T$ to $T$.
\end{proof}

\begin{thm}\label{thm:ToriInA2}
Let $(A,\sigma)$ be a quadratic \'etale algebra with non-trivial involution $\sigma$ and let $(M,\tau)$ be a degree $3$ central simple algebra over $A$ with an involution $\tau$ restricting to $\sigma$ on $A$.

The rational conjugacy classes of maximal tori $T$ in $\SU_{A,\tau}$ are in bijection with pairs $(E,\lambda)$ where $E$ is a cubic \'etale algebra over $k$ such that $E \otimes_k A \injects M$ and $\lambda \in \ker(E^{\times}/N_{E\otimes_k A/E}((E\otimes_k A)^\times)W \rightarrow k^\times/N_{A/k}(A^\times))$ where $W=\Aut_k(E)= (N_{\SU_{A,\tau}}(T)/T)(k)$ are the $k$ rational points of the Weyl group of $T$ subject to the additional constraint that at any real place $\nu$ of $k$ where $A$ is ramified and $E$ is totally real, $\lambda$ is totally positive if and only if $\SU_{A,\tau}(k_\nu)$ is compact.

The torus $T$ is isomorphic to $T_{E\otimes_k A,\sigma,N}$, its points over a ring $R$ are:
\[ T_{E\otimes_k A,\sigma,N}(R) = \{ x\in (E\otimes_k A \otimes_k R)^\times \mid x\sigma(x) = 1 \text{ and } N_{E\otimes_k A/A}(x) = 1 \}. \]

Fixing any embedding $E\injects A$, for which $\tau$ restricts to $\sigma$ on  $E\injects A$, the involutions $\tau_\lambda(x) = \lambda^{-1}\tau(x)\lambda$ satisfies $(M,\tau) \simeq (M,\tau_\lambda)$ and thus $\SU_{A,\tau}\simeq \SU_{A,\tau_\lambda}$, the image of $E\injects A$ under the different embeddings give the different conjugacy classes of $T$ as we vary $\lambda$.
\end{thm}
This is a specialization of the results of \cite{FioriRoe}
 to the case of Hermitian spaces of rank $3$, the case of pure inner forms which we shall actually use in the sequel is also covered in \cite[Ex. 6.129]{WalkerThesisTori}.

\begin{rmk}
In the cases that shall be relevant in the sequel, the algebra $M$ shall always be a matrix algebra, in which case, the condition $E\otimes_k A \injects M$ is automatic.
Moreover, in this setting, we may concretely view $\SU_{A,\tau_\lambda}$ as the isomotries of the Hermitian space with Hermitian form:
\[ \Tr_{E\otimes_k A/A}(\lambda\delta x\sigma(y)) \]
where $\delta\in k^\times/N_{A/k}(A^\times)$ is the discriminant of the Hermitian space with which we are working.
\end{rmk}

Putting all these pieces together we obtain our first formulation of the main result.
\begin{thm}\label{thm:ClassifyToriG2}
Fix an octonion algebra $\bO$ over $k$ and the group $G_k=\Aut_k(\bO)$.
The $G_k(k)$-conjugacy classes of maximal tori $T_k$ in $G_k$ are in bijection with triples $(A,E,\lambda)$ consisting of a quadratic \'etale algebra $A$ (with involutions $\sigma$) which embeds in $\bO$, a cubic \'etale algebra $E$ and an element $\lambda\in (E^\times)/N_{E \otimes_k A/E}((E \otimes_k A)^\times)\Aut_k(E)$ such that the $A$-Hermitian space $E\otimes_k A$ of dimension $3$ with Hermitian form:
\[ \Tr_{E\otimes_k A/A}( \lambda x \sigma(y) ) \]
has discriminant $1$ and is positive definite (respectively indefinite) at all the real places of $k$ where the octonion algebra is definite (respectively split).
Note that the form is positive definite if and only if $\lambda$ is totally positive and $E$ is totally real. Moreover, the discriminant of the form is $N_{E\otimes_k A/A}(\lambda)\delta_{E/k}$.

The torus $T_k$ associated to this data is precisely $T_{E\otimes_k A,\sigma,N}$ whose points over $R$ are: 
\[ T_{E\otimes_k A,\sigma,N}(R) = \{ x\in (E\otimes_k A \otimes_k R)^\times \mid x\sigma(x) = 1 \text{ and } N_{E\otimes_k A/A}(x) = 1 \} \]
where $\sigma$ is induced from the non-trivial automorphism of $A$.
\end{thm}

\begin{rmk}\label{rem:renormalize}
Replacing $\lambda$ by $\delta_{E/k}\lambda$ we see that satisfying the discriminant condition is always possible. Moreover, as the discriminant of the form $\Tr_{E\otimes_k A/A}( \lambda x \sigma(y) ) $ is precisely $N_{E/k}(\lambda)\delta_{E/k}$
this replacement also would allow us to instead consider $\lambda$ in the kernel of the map 
\[ (E^\times)/N_{E \otimes_k A/E}((E \otimes_k A)^\times)\Aut_k(E) \overset{N_{E/k}}\longrightarrow (k^\times)/N_{A/k}(A^\times) \]
subject to the signature conditions.
We see that locally the kernel of the above map is non-trivial only at places of $k$ which split in $E$ for which the corresponding places in $A$ do not split.
For the purpose of this renormalization and its effect on the following proposition it is useful to note that $\delta_{E/k}$ is positive at a real $\nu$ of $k$ if and only if $\nu$ does not ramify in $E$.
\end{rmk}

\begin{prop}\label{prop:ClassifyToriG2R}
We have the following restrictions on the algebras $A$ and $E$ and the signature of $\lambda$ based on the structure of $\bO$ at each real place $\nu$.
The conditions can be summarized as follows:
\begin{itemize}
\item If $\bO$ is definite at $\nu$ then $A$ is a CM-algebra and $E$ is totally real. Moreover, $\lambda$ is positive at all the real places of $E$ over $\nu$.
\item If $\bO$ is indefinite at $\nu$ then $A$ is arbitrary and $E$ is arbitrary. Furthermore, 
\subitem If $A$ is CM and $E$ is totally real then $\lambda$ is positive at a unique place.
\subitem If $A$ is CM and $E$ is not totally real then $\lambda$ is negative at the unique real place (using the normalization of Theorem \ref{thm:ClassifyToriG2}) or positive at the unique real place (using the normalization of Remark \ref{rem:renormalize}).

\subitem If $A$ is totally real then the choice of $\lambda$ at $\nu$ is irrelevant (the norm map from $E\otimes_k A$ to $E$ is surjective at $\nu$).
\end{itemize}
\end{prop}
These conditions are immediate from the structure of the trace form $\Tr_{E\otimes_k A/A}( \lambda x \sigma(y) )$.

As in the previous remark, by replacing $\lambda$ by $\delta_{E/k}\lambda$ we obtain that the following corollary.
\begin{cor}\label{cor:ClassifyToriG2R}
The $G_k(k)$-conjugacy classes of maximal tori $T_k$ in $G_k$ are in bijection with triples $(A,E,\lambda)$
where $A$ and $E$ are respectively quadratic and cubic extensions of $k$, the element $\lambda$ is in the kernel of the map:
\[ (E^\times)/N_{E \otimes_k A/E}((E \otimes_k A)^\times)/\Aut_k(E) \overset{N_{E/k}}\longrightarrow (k^\times)/N_{A/k}(A^\times) \]
and such that the triple $(A,E,\lambda)$ satisfies the conditions of Proposition \ref{prop:ClassifyToriG2R}.
\end{cor}

\begin{rmk}
In order to account for the action of $\Aut_k(E)$ on the options for rational conjugacy classes we observe that for a cubic \'etale extension this group is either $\Sigma_3$ is $E$ is totally split, $C_3$ if $E$ is a cyclic field extension, $C_2$ if $E\simeq \Delta\times k$ or trivial otherwise ($E$ is a field extension but not a Galois extension).

For a localization $k_\nu$ where $A_\nu$ is not split (so that the kernel is not trivial to begin with), we have the following cases:
\begin{itemize}
\item $\Aut_k(E) \simeq \Sigma_3$ then there are $2$ orbits of $\Sigma_3$, the orbit of the trivial element and the orbit of non-trivial elements.
\item $\Aut_k(E) \simeq C_3$, and $E_\nu$ splits, then there are $2$ orbits of $C_3$, the orbit of the trivial element and the orbit of non-trivial elements.
\item $\Aut_k(E) \simeq C_3$, and $E_\nu$ does not split, then the kernel is already trivial, and there is a unique orbit.
\item $\Aut_k(E) \simeq C_2$, and $E_\nu$ splits, then the kernel has $4$ elements, and they are in $3$ orbits under $\Aut_k(E)$.
\item $\Aut_k(E) \simeq C_2$, and $E_\nu$ does not split, then the kernel has $2$ elements, and they are in $2$ orbits under $\Aut_k(E)$.
\item $\Aut_k(E) \simeq \{1\}$, and $E_\nu$ does not split, then the kernel is already trivial, and there is a unique orbit.
\item $\Aut_k(E) \simeq \{1\}$, and $E_\nu$ splits completely, then the kernel has $4$ elements, and there are $4$ orbits.
\item $\Aut_k(E) \simeq \{1\}$, and $E_\nu$ splits partially, then the kernel has $2$ elements, and there are $2$ orbits.
\end{itemize}
Note that for a global field, $\Aut_k(E)$ is acting on the global points of $(E^\times)/N_{E \otimes_k A/E}((E \otimes_k A)^\times)$ and not each localization separately. That is, it acts diagonally on the adelic points.
\end{rmk}

\section{Group Cohomology Interpretation}
\label{sec:groupcohom}

In this final section we describe the connection between the concrete description of the previous section and that which would arise via Galois cohomology.
For the purpose of this section fix $k$, an octonion algebra $\bO$ over $k$, the group $G_k$, a maximal torus $T_k$ and the unique A2 subgroup $H_k$ containing $T_k$.

\begin{thm}
\label{thm:gcm}
Fix a semisimple algebraic group $G$ defined over $k$, a subgroup $H$ defined over $k$ and let $N$ denote the normalizer of $H$ in $G$.
Then the kernel of the map:
\[ H^1(\Gal(\overline{k}/k), N(\overline{k}) ) \rightarrow H^1(\Gal(\overline{k}/k), G(\overline{k}) ) \]
classifies the $G(k)$-conjugacy classes of $G(\overline{k})$-conjugates of $H$ which are defined over $k$.
\end{thm}
See \cite[Lemma 6.2]{ReederElliptic}.

\begin{rmk}
Via twisting the set $H^1(\Gal(\overline{k}/k), N(\overline{k}) )$ can be seen to classify $G(\overline{k})$-conjugates of $H$ which happen to be defined over $k$ appearing in pure inner forms of $G$.
For the groups $G$ of type G2, the set of pure inner forms coincides with the set of forms. Thus the goal of this section is to relate the results of the previous section to a description of the map:
\[ H^1(\Gal(\overline{k}/k), N(\overline{k}) ) \rightarrow H^1(\Gal(\overline{k}/k), G(\overline{k}) ). \]
\end{rmk}

\begin{prop}
\label{prop:normalizer}
We have the following:
\begin{itemize}
\item
The normalizer $N_k$ of $T_k$ contains an outer automorphism of $H_k$.
\item
There is a group $\tilde{A}_k$ together with maps:
\[ H_k \injects \tilde{A}_k \surjects \Aut(H_k) \]
where the kernel of the map $\tilde{A}_k\rightarrow  \Aut(H_k)$ is the image of the center $Z(H_k)$ of $H_k$.
\end{itemize}
\end{prop}
\begin{proof}
Let $x$ be a $k$-rational point of $\bO^0$ stabilized by $H_k$.
Consider the element $g\in G_k$ which takes $x$ to $-x$, and fixes $y,z$.
This element then fixes $yz$ and takes each of $x,xy,xz,(xy)z$, respectively, to $-x,-xy,-xz,(-xy)z$.
Such a map is induced by conjugation on the underlying Hermitian space $A_x^\perp$, hence gives the outer automorphism of $H_k$.
By base change to the algebraic closure, we can check that the outer automorphism acts as inversion on the torus, in particular all tori in $H_k$ are self adjoint and this action induces an involution of $T_k$.
\end{proof}

Define  $\tilde{A}_k$ to be the normalizer of $H_k$ in $G_k$. It is clear that $\tilde{A}_k$ is generated by $H_k$ and the element $g$ as in the proposition.

\begin{prop}\label{prop:CohomFormA2inG2}
The forms of $H_k$ which may embed in some form of $G_k$ are the ``pure outer forms" of $H_k$, namely they arise from
\[ H^1(\Gal(\overline{k}/k), \tilde{A}_k(\overline{k})) \rightarrow H^1(\Gal(\overline{k}/k), \Aut_{\overline{k}}(H_k) ). \]

The forms of $H_k$ which embed in $G_k$ are additionally in the kernel of the map:
\[ H^1(\Gal(\overline{k}/k), \tilde{A}_k(\overline{k}) ) \rightarrow  H^1(\Gal(\overline{k}/k), G_k(\overline{k}) ). \]
\end{prop}
\begin{proof}[Remarks on Proof]
This is a restatement of Theorem \ref{thm:gcm}, but notice that the first statement agrees with our results from Theorem \ref{thm:SUclassification} in light of Proposition \ref{prop:FormsA2}.
The fact that rational conjugacy is equivalent to rational isomorphism is clear in the observation that
$H^1(\Gal(\overline{k}/k), \tilde{A}_k(\overline{k}))$ classifies both.

In light of this, we may reinterpret Corollary \ref{cor:SUoverR} and Proposition \ref{prop:FormsG2} as giving a concrete description of the second map. 
Thus proving the second statement amounts to showing that 
certain explicit cocycles in $H^1(\Gal(\bC/\bR), \tilde{A}_k(\bC) )$ split in $H^1(\Gal(\bC/\bR), G_k(\bC) )$.
\end{proof}

\begin{prop}
\label{prop:gchk}
Let $M_k$ be the normalizer in $H_k$ of $T$.
The forms of $T$ which embed in $H_k$ are those in the image of:
\[  H^1(\Gal(\overline{k}/k), M_k(\overline{k})) \rightarrow H^1(\Gal(\overline{k}/k), \Aut_{\overline{k}}(T)) \]
and in the kernel of:
\[ H^1(\Gal(\overline{k}/k), M_k(\overline{k})) \rightarrow  H^1(\Gal(\overline{k}/k), H_k(\overline{k}) ).\]
\end{prop}
The concrete relation between this and our description is found in either \cite{FioriRoe} or \cite[Ex. 6.129]{WalkerThesisTori}.

\begin{rmk}
There is an exact sequence:
\[ 1 \rightarrow T_k(\overline{k}) \rightarrow M_k(\overline{k}) \rightarrow \Sigma_3 \rightarrow 1. \]
Where $\Sigma_3$ denotes the symmetric group on three elements.
It is known that $H^1(\Gal(\overline{k}/k), \Sigma_3)$ classifies degree $3$ \'etale algebras over $k$.
\end{rmk}

\begin{prop}
Let $N_k$ denote the normalizer of $T$ in $G_k$.
The forms of $T$ which embed in $G_k$ are those in the image of:
\[  H^1(\Gal(\overline{k}/k), N_k(\overline{k})) \rightarrow H^1(\Gal(\overline{k}/k), \Aut_{\overline{k}}(T)) \]
and in the kernel of:
\[ H^1(\Gal(\overline{k}/k), N_k(\overline{k})) \rightarrow  H^1(\Gal(\overline{k}/k), G_k(\overline{k}) )\]
\end{prop}
\begin{proof}[Remarks on proof]
This is a restatement of Theorem \ref{thm:gcm}, but notice that the statement follows by combing the results from Theorems \ref{thm:ToriInG2} and \ref{thm:ToriInA2} and  Propositions \ref{prop:CohomFormA2inG2} and \ref{prop:gchk}.

Moreover, there is an exact sequence:
\[ 1 \rightarrow T_k(\overline{k}) \rightarrow N_k(\overline{k}) \rightarrow \Sigma_2 \times \Sigma_3 \rightarrow 1 \]
$H^1(\Gal(\overline{k}/k), \Sigma_2 \times \Sigma_3)$ describes pairs of quadratic \'etale and cubic \'etale algebras.
The fact that rational conjugacy and rational isomorphism are not equivalent is a result of the fact that the map from $N_k(\overline{k}) \rightarrow \Aut(T)$ has kernel $T$ and
\[ H^1(\Gal(\overline{k}/k), T(\overline{k})) \]
is not necessarily trivial. This cohomology group is calculated via the exact sequence:
\[ H^0(\Gal(\overline{k}/k), T_E(\overline{k})) \rightarrow H^0(\Gal(\overline{k}/k), T_{E^\sigma}(\overline{k})) \rightarrow  H^1(\Gal(\overline{k}/k), T(\overline{k})) \rightarrow H^1(\Gal(\overline{k}/k), T_E(\overline{k})) \]
and the observation that by Hilbert's Theorem 90, the group $ H^1(\Gal(\overline{k}/k), T_E(\overline{k})) = \{ 1\}$.
Completing the computations agrees precisely with what is obtained from Proposition \ref{prop:ToriInGroups2} and Theorem \ref{thm:ClassifyToriG2}.

For the second claim notice that we have maps:
\[ H^1(\Gal(\overline{k}/k), N_k(\overline{k})) \rightarrow H^1(\Gal(\overline{k}/k), \tilde{A}_k(\overline{k}))  \rightarrow  H^1(\Gal(\overline{k}/k), G_k(\overline{k}) ).\]
The second claim is thus made equivalent to Proposition \ref{prop:gchk} and Corollary \ref{cor:SUoverR}.
\end{proof}

\section{Concluding Remarks}

We have been able to accomplish our three important goals:
\begin{enumerate}
\item We have given a concrete description and classification of the rational conjugacy classes of maximal tori in groups of type G2.
\item We have given a concrete description and classification of the rational conjugacy classes of simply connected A2 subgroups of G2.
\item Finally, we have been able to relate these concrete descriptions to the more abstract cohomological descriptions that were already available.
\end{enumerate}

It remains a goal to give concrete classifications for the rational conjugacy classes of maximal tori in all semi-simple and reductive groups.
The families of classical groups admit systematic approaches, and the work of \cite{Fiori1,FioriRoe} already handles many (though not all) of these.

The group G2 being the simplest of the exceptional groups was a natural candidate to be the first exceptional group to consider.
Though the approach here seems ad-hoc given the heavy reliance on the structure of $\bO$, it is expected that similar methods will work for the group $F4$ (this is an ongoing project) and with other ideas it is hoped that the cases of $E6,E7$ and $E8$ may be handled.

\section*{Acknowledgements}
I would like to thank Eva Bayer-Fluckiger, Skip Garibaldi, and Andrei Rapinchuk for various useful comments and suggestions.
I would like to thank Victoria de Quehen for her help in editing various drafts of this paper.

{}\ifx\XMetaCompile\undefined

\providecommand{\MR}[1]{}
\providecommand{\bysame}{\leavevmode\hbox to3em{\hrulefill}\thinspace}
\providecommand{\MR}{\relax\ifhmode\unskip\space\fi MR }
\providecommand{\MRhref}[2]{  \href{http://www.ams.org/mathscinet-getitem?mr=#1}{#2}
}
\providecommand{\href}[2]{#2}

\end{document}
{}\fi